\documentclass[11pt,reqno]{amsart}
\usepackage{amssymb,mathrsfs,graphicx}
\usepackage{ifthen}
\usepackage{colortbl}
\definecolor{black}{rgb}{0.0, 0.0, 0.0}
\definecolor{red}{rgb}{1.0, 0.5, 0.5}
\provideboolean{shownotes} 
\setboolean{shownotes}{true} 
\newcommand{\margnote}[1]{
\ifthenelse{\boolean{shownotes}}%
{\marginpar{\raggedright\tiny\texttt{#1}}}%
{}%
}
\newcommand{\hole}[1]{
\ifthenelse{\boolean{shownotes}}%
{\begin{center} \fbox{ \rule {.25cm}{0cm} \rule[-.1cm]{0cm}{.4cm}
\parbox{.85\textwidth}{\begin{center} \texttt{#1}\end{center}} \rule
{.25cm}{0cm}}\end{center}} {} }

\def\d{\,\mathrm{d}}
\def\tot#1#2{\frac{\d #1}{\d #2}}

\def\eps{\varepsilon}

\title[Hydrodynamic Cucker-Smale model]{Hydrodynamic Cucker-Smale model with normalized communication weights and time delay}

\author[Choi]{Young-Pil Choi}
\address[Young-Pil Choi]{\newline Department of Mathematics
    \newline  Inha University, Incheon 402-751, Republic of Korea}
\email{ypchoi@inha.ac.kr}

\author[Haskovec]{Jan Haskovec}
\address[Jan Haskovec]{\newline Computer, Electrical and Mathematical Sciences \& Engineering
    \newline King Abdullah University of Science and Technology, 23955 Thuwal, KSA}
\email{jan.haskovec@kaust.edu.sa}

\numberwithin{equation}{section}

\newtheorem{theorem}{Theorem}[section]
\newtheorem{lemma}{Lemma}[section]

\newtheorem{remark}{Remark}[section]
\newtheorem{definition}{Definition}[section]
\newtheorem{assumption}{Assumption}
\newtheorem{convention}{Convention}

\def\({\begin{eqnarray}}
\def\){\end{eqnarray}}
\def\[{\begin{eqnarray*}}
\def\]{\end{eqnarray*}}

\def\Norm#1{\left\| #1 \right\|}

\newcommand{\R}{\mathbb R}
\newcommand{\om}{\Omega}
\newcommand{\ls}{\lesssim}

\newcommand{\mc}{\mathcal C}

\newcommand{\bq}{\begin{equation}}
\newcommand{\eq}{\end{equation}}
\newcommand{\e}{\varepsilon}
\newcommand{\lt}{\left}
\newcommand{\rt}{\right}

\newcommand{\pa}{\partial}

\newcommand{\wt}{\widetilde}
\newcommand{\Om}{\Omega}
\newcommand{\ppsi}{\psi_{t,\tau}[\eta]}

\def\Lyap{\mathcal{L}}

\def\N{\mathbb{N}}
\def\d{\mathrm{d}}

\def\wpsi{\widetilde\psi}

\begin{document}
\allowdisplaybreaks

\date{\today}


\begin{abstract}
We study a hydrodynamic Cucker-Smale-type model with time delay in communication and information processing, in which agents interact with each other through normalized communication weights. The model consists of a pressureless Euler system with time delayed non-local alignment forces. We resort to its Lagrangian formulation and prove the existence of its global in time classical solutions. Moreover, we derive a sufficient condition for the asymptotic flocking behavior of the solutions. Finally, we show the presence of a critical phenomenon for the Eulerian system posed in the spatially one-dimensional setting.
\end{abstract}

\maketitle \centerline{\date}

\tableofcontents

%
%
%
%
\section{Introduction}
We study the existence of global classical solutions and asymptotic behavior of the
following system of pressureless Euler equations with time delayed non-local alignment forces:
\(   \label{Eul1}
\pa_t \rho_t + \nabla \cdot (\rho_t u_t) &=& 0, \\
\pa_t (\rho_t u_t) + \nabla \cdot (\rho_t u_t \otimes u_t) &=&
   \rho_t\frac{\int_{\R^d} \psi(x-y) \rho_{t-\tau}(y)u_{t-\tau}(y)\,dy}{\int_{\R^d} \psi(x-y) \rho_{t-\tau}(y)\,dy} - \rho_t u_t, \label{Eul2}
\)
for $t\geq 0$ and $x\in\R^d$ with $d\in\N$ the space dimension.
The constant $\tau  \geq 0$ denotes the fixed delay in communication and information processing.
Here and in the sequel we denote by the subscript $\{\cdot\}_t$ the time-dependence of the respective variable.
The \emph{influence function} $\psi: \R^d \to \R_+$ satisfies the following set of assumptions:

\begin{assumption}\label{ass:psi}
The influence function $\psi$ is continuous and continuously differentiable on $\R^d$
with uniformly bounded derivatives up to order $\ell \in \N \cup \{0\}$. Moreover, it is radially symmetric, i.e., there exists a function
$\wpsi: [0,+\infty) \to (0,+\infty)$ such that
\[
   \psi(x) = \wpsi(|x|)\qquad\mbox{for all }x\in\R^d.
\]
The function $\wpsi$ is nonincreasing, positive and uniformly bounded on $[0,+\infty)$.
Without loss of generality, we assume $\wpsi(0)=1$.
\end{assumption}
Let us note that the (rescaled) influence function introduced in the seminal papers
by Cucker and Smale \cite{CS1, CS2}, namely
\bq\label{psi_cs}
\psi(x) = \frac{1}{(1 + |x|^2)^\beta}
\eq
with $\beta \geq 0$, satisfies the above set of assumptions.

Associated to the fluid velocity $u_t$, we define the characteristic flow $\eta_t: \R^d \to \R^d$ by
\(\label{eta_flow}
   \frac{\d \eta_t(x)}{\d t} = u_t(\eta_t(x)) \quad\mbox{for  } t \geq -\tau, \qquad \mbox{subject to } \eta_{0}(x) = x \in \R^d.
\)
Note that, emanating from $\eta_{0}(x) = x$ at $t=0$, we solve \eqref{eta_flow} both forward and backward in time to obtain the characteristics for $t \geq -\tau$. Let us denote the time-varying set $\Omega_t := \{x \in \R^d\,:\, \rho_t(x) \neq 0 \}$ for given initially bounded open set $\Omega_0$. 

The system \eqref{Eul1}--\eqref{Eul2} is considered subject to the initial data
\(\label{EulIC}
   (\rho_s(x), u_s(x)) = (\bar \rho_s(x),\bar u_s(x)) \qquad\mbox{for } (s,x) \in [-\tau,0]\times \Omega_s.
\)
Since the total mass is conserved in time, without loss of generality, we may assume that $\rho_t$ is a probability density function, i.e., $\|\rho_t \|_{L^1} = 1$ for all $t \geq 0$ by assuming $\|\rho_0\|_{L^1} = 1$. 



The system \eqref{Eul1}--\eqref{Eul2} can be formally derived
from the kinetic Cucker-Smale type model, introduced and studied in \cite{Choi-Haskovec},
\( \label{kin_mt}
   \partial_t f_t + v \cdot \nabla_x f_t + \nabla_v \cdot (F[f_{t-\tau}]f_t) = 0 \qquad\mbox{for } (x,v) \in \R^d \times \R^d,
\)
with
\[
   F[f_{t-\tau}](x,v) := \displaystyle \frac{\int_{\R^{2d}} \psi(x-y)(w-v)f(y,w,t-\tau)\,\d y\d w}{\int_{\R^{2d}} \psi(x-y)f(y,w,t-\tau)\,\d y\d w}.
\]
Here the one-particle distribution function $f_t=f_t(x,v)$ is
a time-dependent probability measure on the phase space $\R^d\times\R^d$,
describing the probability of finding a particle at time $t\geq 0$ located at $x\in\R^d$
and having velocity $v\in\R^d$.
The normalization in the expression for the interaction force $F[f_{t-\tau}]$
is similar to the one introduced in \cite{MT}.
We define the mass and momentum densities by
\[
   \rho_t(x) := \int_{\R^d} f_t(x,v) \,\d v,\qquad \rho_t(x)u_t(x) := \int_{\R^d} v f_t(x,v) \,\d v.
\]
Then, \eqref{Eul1} is obtained directly by integrating the Vlasov equation \eqref{kin_mt}
with respect to $v$, while \eqref{Eul2} follows from taking the first-order moment with respect to $v$
and adopting the monokinetic closure $f_t(x,v) = \rho_t(x) \delta(v-u_t(x))$.

In \cite{Choi-Haskovec} we proved the global existence and uniqueness of measure-valued solutions of \eqref{kin_mt}.
Moreover, we provided a stability estimate in terms of the Monge-Kantorowich-Rubinstein distance,
and, as a direct consequence, an asymptotic flocking result for the kinetic system. 
Here we shall extend these results for the hydrodynamic system \eqref{Eul1}--\eqref{Eul2}.
In particular, we shall study the existence of global classical solutions, their asymptotic behavior
(commonly referred to as \emph{flocking}) and propagation of smoothness.
We refer to \cite{CCP, CHL} for a recent overview of emergent dynamics of the Cucker-Smale model
and its variants.

\begin{definition}\label{def:classical}
We call $(\rho_t, u_t)$ a classical solution of the system \eqref{Eul1}--\eqref{Eul2} on $[0,T)$,
subject to the initial datum \eqref{EulIC}, if $\rho_t$ and $u_t$ are continuously differentiable 
functions on the set $\{(t,x)\in [0,T)\times\Omega_t\}$, the characteristics $\eta_t = \eta_t(x)$
defined by \eqref{eta_flow} are diffeomorphisms for all $t\in [0,T)$,
and $\rho_t$ and $u_t$ satisfy the equations \eqref{Eul1}--\eqref{Eul2}
pointwise in $\{(t,x)\in [0,T)\times\Omega_t\}$, with the initial datum \eqref{EulIC}.
The time derivative at $t=0$ has to be understood as a one-sided derivative.
\end{definition}

The aim of this paper is to prove the existence of classical solutions of the system \eqref{Eul1}--\eqref{Eul2} and to study their asymptotic behavior for large times; in particular, we shall give sufficient conditions that lead to \emph{asymptotic flocking} in the sense of the definition of Cucker and Smale \cite{CS1, CS2}, see statement \eqref{statement2} of Theorem \ref{thm:flocking}. We carry out this program by resorting to the Lagrangian formulation of the system \eqref{Eul1}--\eqref{Eul2}. This is derived in Section \ref{sec:main}, where we also state our main results. For global existence of solutions, we provide two different strategies: first one is based on Cauchy-Lipschitz theory, which is usually used for constructing a solution to systems of differential equations. This gives the global existence of smooth solutions for the Lagrangian system with only continuous initial datum, any further smoothness or smallness of the initial datum are not required. However, unfortunately, this argument cannot be applied to the case $\tau = 0$, i.e., no time delay. Furthermore, it is not that clear how to shift the existence result of Lagrangian system to the Eulerian system. On the other hand, the second strategy is based on the energy method combined with the large-time behavior estimate of solutions. This also does not require any smallness assumption for the initial data and provides the initial regularity persists globally in time. However, compared to the first strategy, we need additional assumptions used for the large-time behavior estimate, see Theorem \ref{thm_main}. Despite such assumptions, this strategy can be directly applied to the case of no time delay, $\tau = 0$, and the global existence of classical solutions to the Eulerian system \eqref{Eul1}-\eqref{Eul2} if we further assume that the initial datum are small enough. It is worth mentioning that the smallness assumption on the datum is not needed to construct the global-in-time classical solutions for the Lagrangian system \eqref{Lagr1}-\eqref{Lagr2}. 

The rest of the paper is organized as follows. As mentioned above, we discuss the derivation of the Largangian system from the Eulerian system \eqref{Eul1}-\eqref{Eul2} and present our main results on the global existence of solutions, the large-time behavior  of solutions, and the critical phenomena for the Eulerian system in the one dimensional case. Continuous solutions of the Lagrangian formulation are constructed in Section \ref{sec:existence}. In Section \ref{sec:flocking} we study the asymptotic flocking behavior of the solutions, and in Section \ref{sec:classical} we prove the existence of global classical solutions of the system \eqref{Eul1}--\eqref{Eul2} in the sense of Definition \ref{def:classical}. Finally, in Section \ref{sec:critical} we show the presence of a critical phenomenon for the Eulerian system posed in the spatially one-dimensional setting.

%
%
\section{Lagrangian formulation and main results}\label{sec:main}
In the sequel we shall work with the Lagrangian formulation of the system \eqref{Eul1}--\eqref{Eul2}.
For $x\in\Omega_0$ we introduce the functions
\(  \label{hv}
   h_t(x):= \rho_t(\eta_t(x)),\qquad v_t(x) := u_t(\eta_t(x)),
\)
where $\eta_t$ is the characteristic flow defined in \eqref{eta_flow}.
Then, we formally rewrite the system \eqref{Eul1}--\eqref{Eul2} as
\begin{align}\label{Lagr1}
\begin{aligned}
\frac{\d \eta_t(x)}{\d t} &= v_t(x),\\
\frac{\d v_t(x)}{\d t} 
&= \frac{\int_{\Om_0} \psi(\eta_t(x) - \eta_{t-\tau}(y)) \rho_0(y)v_{t-\tau}(y)\,dy}{\int_{\Om_0} \psi(\eta_t(x) - \eta_{t-\tau}(y)) \rho_0(y)\,dy} - v_t(x),
\end{aligned}
\end{align}
and
\(   \label{Lagr2}
   h_t(x) = \rho_0(x) \det(\nabla \eta_t(x))^{-1}.
\)
We refer to \cite{DS, HKK15} for details. The system \eqref{Lagr1} is subject to the initial datum
\(   \label{LagrIC}
   v_s(x) := \bar u_s(\eta_s(x))\qquad\mbox{for } s\in[-\tau,0],\; x\in\Omega_0.
\)
Classical solutions of the system \eqref{Lagr1}--\eqref{LagrIC} on $[0,T)\times\Omega_0$
are defined analogously to Definition \ref{def:classical}.
Then, as long as the characteristic flow $\eta_t$ given by \eqref{eta_flow}
is a diffeomorphism between $\Omega_0$ and $\Omega_t$,
\eqref{hv} defines an equivalence between the classical solutions
of the Eulerian system \eqref{Eul1}--\eqref{EulIC}
and the classical solutions of the Lagrangian formulation
\eqref{Lagr1}--\eqref{LagrIC}.

Observe that the equation \eqref{Lagr2} for the mass density $h_t$ is decoupled from the system
\eqref{Lagr1} for $(\eta_t, v_t)$. Therefore, our first main result establishes the global in time existence
of solutions of \eqref{Lagr1}. The mass density $h_t$ is then calculated as a post-processing step,
assuming that $\eta_t$ is a diffeomorphism, i.e., that the matrix $\nabla\eta_t$ is invertible.

\begin{theorem}\label{thm:existence}
Let Assumption \ref{ass:psi} be verified and $\tau > 0$.
Suppose that the initial datum $(\eta_s, v_s)\in\mc([-\tau,0]\times\overline\Omega_0)$.
Then there exists a unique global in time solution $\eta_t\in\mc^1([0,\infty); \mc(\overline{\Omega_0}))$,
$v_t\in\mc([0,\infty)\times\overline\Omega_0)$ of the system \eqref{Lagr1}, satisfying
\(   \label{global_v}
   \Norm{v_t}_{L^\infty([0,\infty)\times \Omega_0)} \leq \max_{s\in[-\tau,0]} \Norm{v_s}_{L^\infty(\Omega_0)}.
\)
\end{theorem}
\begin{remark}\label{rmk_21}Applying a bootstrapping argument to Theorem \ref{thm:existence} actually yields the global existence of classical solutions to the Lagrangian system \eqref{Lagr1}.
\end{remark}

Our second result deals with the asymptotic behavior of the solutions of \eqref{Lagr1}
for large times. In particular, we prove that under additional assumptions on the 
initial velocity distribution and the influence function,
the system exhibits the so-called \emph{flocking behavior} \cite{CS1, CS2},
where the velocities converge to a common consensus value,
while the mutual distances stay uniformly bounded.
Let us introduce the notation for the spatial and velocity diameters of the solution,
\(   \label{dXdV}
   d_X(t) := \max_{x,y \in \overline{\Omega}_0} |\eta_t(x) - \eta_t(y)|,
   \qquad d_V(t) := \max_{x,y \in \overline{\Omega}_0} |v_t(x) - v_t(y)|.
\)

\begin{theorem}\label{thm:flocking}
Let Assumption \ref{ass:psi} be verified and $(\eta_t, v_t)\in\mc^1([0,\infty); \mc(\overline{\Omega_0}))$ be a solution of the system \eqref{Lagr1}.
Suppose that the initial datum $v_s\in\mc([-\tau,0]\times\overline\Omega_0)$ and the influence function $\wpsi$ satisfy the following conditions:
\(  \label{R_V}
   \max_{s \in [-\tau,0]}\max_{x \in \overline \Omega_0} |v_s(x)| =: R_V < +\infty,
\)
and
\(   \label{iii}
  d_V(0)  + \int_{-\tau}^0 d_V(s)\d s < \int_{d_X(-\tau) + R_V \tau}^\infty \wpsi(s) \ ds,
\)
with $d_X$ and $d_V$ defined in \eqref{dXdV}.
Then the spatial diameter $d_X$ of the solution of \eqref{Lagr1} is uniformly bounded
and the velocity diameter $d_V$ decays exponentially in time,
\(   \label{statement2}
   \sup_{t\geq 0} d_X(t) < +\infty, \qquad
   d_V(t) \leq \left( \max_{s\in[-\tau,0]} d_V(s) \right) e^{-C t} \quad \mbox{for } t \geq 0,
\)
for a suitable constant $C>0$ independent of time.
\end{theorem}


The assumption \eqref{iii} can be understood, for a fixed integrable influence function $\wpsi$,
as a condition for smallness of the delay $\tau$. Indeed, considering a fixed initial datum with
$d_V(s)\equiv: \bar d_V >0$ constant for $s\in [-\tau,0]$ and $d_X(-\tau)\equiv: \bar d_X\geq 0$,
then \eqref{iii} reads
\[
   (1+\tau) \bar d_V < \int_{\bar d_X + R_v \tau}^\infty \wpsi(s)\,ds.
\]
Clearly, the left-hand side increases with increasing $\tau$, while the right-hand side decreases.
So, generically, it is necessary to choose $\tau$ sufficiently small in order to satisfy the flocking condition.
This is often the case in alignment models with delay, see, e.g., \cite{EHS}.
On the other hand, if the influence function $\wpsi$ has a heavy tail, i.e.,
\[
   \int^\infty \wpsi(s)\,d s = +\infty,
\]
then assumption \eqref{iii} is satisfied for any initial datum and any $\tau \geq 0$,
which is a situation usually called \emph{unconditional flocking}, see, e.g., \cite{CS1, CS2}.
Let us note that if the influence function $\wpsi$ is of the commonly used form
\[
   \wpsi(s) = \frac{1}{(1+s^2)^{\beta}},
\]
then unconditional flocking takes place for $\beta \in [0,1/2]$.
In this case the Assumptions \ref{ass:psi} and \eqref{iii} are satisfied.

Our third and final main result is concerned with the existence and uniqueness
of global classical solutions of the system \eqref{Lagr1}.
This result is based on proving sufficient regularity of the solutions constructed in Theorem \ref{thm:existence},
for which we will need the estimates derived in Theorem \ref{thm:flocking}.
Thus the below result adopts the assumptions of Theorem \ref{thm:flocking}.

\begin{theorem}\label{thm_main}
Let Assumption \ref{ass:psi} be verified with some $\ell > \frac{d}2+1$, and let \eqref{R_V} and \eqref{iii} hold. Moreover, we assume that the initial datum satisfying the regularity: 
\[
(\bar\rho_s, \bar u_s) \in \mc([-\tau,0];H^\ell(\Omega_s)) \times \mc([-\tau,0];H^{\ell+1}(\Omega_s)). 
\]
Then the system \eqref{Lagr1} admits a unique global classical solution
$(\eta_t, v_t) \in \mc^1([-\tau,\infty);H^{\ell+1}(\Omega_0)) \times \mc([-\tau,\infty);H^{\ell+1}(\Omega_0))$.
\end{theorem}

For notational simplicity, we denote by $\|f\|_{L^p}$ the usual $L^p(\Omega_0)$-norm for a function $f(x)$ if there is no confusion, unless otherwise specified.

\begin{remark}
Note that for $\ell > d/2 + 1$ we have the embedding of the Sobolev space $H^\ell(\Omega)$
into the space of continuous functions $\mc^1(\Omega)$. Thus the existence of solutions for the large-time behavior estimate \eqref{statement2} is also justified by Theorem \ref{thm_main}. 
\end{remark}

So far, we established the global regularity of solutions for the Cauchy problem in the Lagrangian coordinates not taking into account the equation \eqref{Lagr2}. In that case, we do not need any smallness assumptions on the initial data, see Theorems \ref{thm:existence} and \ref{thm_main}. However, in order to go back to the Eulerian variables to study the global regularity for the Cauchy problem \eqref{Eul1}-\eqref{Eul2}, the smallness assumption on the initial data is required. In fact, we will show that smooth solutions can be blow up in a finite time when the initial data are not that small, see Theorem \ref{thm_cri} below for details. Note that if the characteristic flow defined in \eqref{eta_flow} is diffeomorphism, i.e., det$\nabla \eta_t > 0$ for all $t \geq 0$, then we can consider the Cauchy problem in the Eulerian coordinates. The theorem below shows that if the initial datum is small enough then the flow $\eta_t$ is indeed a diffeomorphism.



\begin{theorem} \label{thm:Eulerian} Let the same assumptions in Theorem \ref{thm_main} be verified. Moreover, suppose that the initial data $\bar u_s \in \mc([-\tau,0];H^{\ell+1}(\om_s))$ satisfy $\|\nabla \bar u_0\|_{L^2} + \max_{s \in [-\tau,0]}d_V(s) \leq \e$ for sufficiently small $\e > 0$. Then we have the global existence and uniqueness of classical solutions to the system \eqref{Eul1}-\eqref{Eul2}. 
\end{theorem}

Finally, we show that the spatially one-dimensional
version of the system \eqref{Eul1}--\eqref{EulIC} exhibits a critical threshold
in terms of the derivative of the initial datum $\bar u_0$. 
In particular, if $\partial_x \bar u_0(x)$ is negative enough
for some $x\in\Omega_0$, then the corresponding solution
blows up in finite time. This is due to the fact that $\eta_t$
ceases to be a diffeomorphism. The critical threshold phenomena for flocking models are studied in \cite{CCTT, TT}.

\begin{theorem}\label{thm_cri}
Consider the system \eqref{Eul1}--\eqref{Eul2} with $d=1$. Let Assumption \ref{ass:psi} be verified with some $\ell \geq 1$. Moreover, we assume that the influence function $\psi$ satisfies $|\psi'| \leq C|\psi|$ for some positive constant $C$. Let $\overline C = 2CR_V$ with $R_V$ appeared in \eqref{R_V}.
\begin{itemize}
\item
If $\overline C \leq 1$ and $\partial_x u_0(x) \geq - \left(1 + \sqrt{1 - 4\overline C}\right)/2$ for all $x \in \R$,
then system has a global classical solution.
\item
If there exists an $x\in\R$ such that $\pa_x u_0(x) < -\left( 1 + \sqrt{1 + \overline C}\right)/2$,
then the solution blows up in a finite time.
\end{itemize}
\end{theorem}

\begin{remark}It is easy to check that the influence function given in \eqref{psi_cs} satisfies $|\psi'| \leq \beta|\psi|$.
\end{remark}
\begin{remark}The condition $|\psi'| \leq C|\psi|$ can be replaced by the assumption for the large-time behavior estimate \eqref{R_V} and \eqref{iii}. In fact, in that case, the constant $\overline C$ is given by
\[
\overline C = \frac{R_V \Norm{\psi'}_{L^\infty(0,\infty)}}{\psi(d_X^M)} \left( 1 + \frac{1}{\psi(d_X^M)} \right),
\] 
where $d_X^M = \sup_{t\geq -\tau} d_X(t)$. 
\end{remark}

\begin{convention}
In the rest of the paper, generic, not necessarily equal, constants will be denoted by $C$.
\end{convention}

%
%
\section{Existence of solutions for the Lagrangian system - proof of Theorem \ref{thm:existence}} \label{sec:existence}
We start by proving the following technical Lemma.

\begin{lemma}\label{lem:growth}
Let $u=u(t)$ be a nonnegative, continuous and piecewise $\mc^1$-function
satisfying the inequality
\(   \label{growth_ineq}
   \tot{}{t} u(t) \leq C_1 + C_2 \int_0^t u(s) \d s\qquad\mbox{for almost all } t > 0,
\)
with some constants $C_1, C_2 > 0$.
Then
\(  \label{growth_claim}
   u(t) \leq \left(u(0) + \frac{C_1}{\sqrt{C_2}} \right) e^{\sqrt{C_2}\, t} \qquad\mbox{for all } t > 0.
\)
\end{lemma}

\begin{proof}
We integrate \eqref{growth_ineq} on $(0,t)$,
\(  \label{growth_int}
   u(t) \leq u(0) + C_1 t + C_2 \int_0^t \int_0^s u(r) \d r\d s.
\)
Let us denote
\[
   T := \sup \{ t>0;\, \eqref{growth_claim} \mbox{ holds on } [0,t] \}.
\]
Since $C_1>0$ and due to the continuity of $u=u(t)$, we have $T>0$.
For contradiction, assume that $T<+\infty$.
Then, obviously, $u(T) = \left(u(0) + \frac{C_1}{\sqrt{C_2}} \right) e^{\sqrt{C_2}\, T}$,
and inserting this into \eqref{growth_int} gives
\[
   \left(u(0) + \frac{C_1}{\sqrt{C_2}} \right) e^{\sqrt{C_2}\, T}
   &\leq&
   u(0) + C_1 T + C_2 \int_0^T \int_0^s u(r) \d r\d s   \\
   &\leq&
   u(0) + C_1 T + \left(u(0) + \frac{C_1}{\sqrt{C_2}} \right) \left( e^{\sqrt{C_2}\, T} - 1- T\sqrt{C_2} \right).
\]
This further implies
\[
   0 \leq - \frac{C_1}{\sqrt{C_2}} - T \sqrt{C_2} u(0),
\]
a contradiction to the assumption $u(0) \geq 0$ and the positivity of $C_1$ and $C_2$.
\end{proof}

Before proceeding with the proof of Theorem \ref{thm:existence},
let us make a remark about the structure of the system \eqref{Lagr1}.
Although the equation for $v_t$ formally is an integro-differential equation,
observe that only $v_{t-\tau}(y)$ appears in the integrand on its right-hand side,
while $\eta_t$ appears as $\eta_t(x)$ 
(while the integration is performed with respect to $y$).
Consequently, employing the method of steps, see, e.g., \cite{Smith},
we shall prove the existence of solutions inductively on the intervals
$[0,\tau]$, $[\tau, 2\tau]$, etc. On each of these intervals, \eqref{Lagr1}
is merely a family of ordinary differential equations for $(\eta_t, v_t)$,
parametrized by $x\in\Omega_0$, with $v_{t-\tau}$ taken from the previous step.
Therefore, we can employ the classical Cauchy-Lipschitz theorem for proving local in time existence
of solutions for each fixed $x\in\Omega_0$.
A suitable a priori estimate, uniform in $t$ and $x$, will then provide the global in time existence.

\begin{proof}[Proof of Theorem \ref{thm:existence}]
We proceed inductively on time intervals of length $\tau > 0$.
For a prescribed $(\eta_s,v_s)\in\mc([-\tau,0]\times\overline\Omega_0)$ and for $t\in(0,\tau)$, we denote
\[
    F_t[\eta] := \int_{\Omega_0} \psi(\eta - \eta_{t-\tau}(y)) \rho_0(y)v_{t-\tau}(y)\d y,\qquad
    G_t[\eta] := \int_{\Omega_0} \psi|\eta - \eta_{t-\tau}(y)) \rho_0(y)\d y.
\]
Then, the system \eqref{Lagr1} is written as
\(
   \frac{\d \eta_t(x)}{\d t} &=& v_t(x),  \label{ODEs1}\\
   \frac{\d v_t(x)}{\d t} &=& \frac{F_t[\eta_t(x)]}{G_t[\eta_t(x)]} - v_t(x),  \label{ODEs2}
\)
which is a family of ODE systems on $(0,t)$, parametrized by $x\in\Omega_0$, subject to the initial datum
$\eta_0(x) = x$ and $v_0(x)$ given by the value of $v_s(x)$ at $s=0$.
For any fixed $x\in\Omega_0$, we will show local in time existence of solutions \eqref{ODEs1}--\eqref{ODEs2} employing the classical
Cauchy-Lipschitz theorem.

We only need to prove that the expression $\frac{F_t[\eta]}{G_t[\eta]}$ is locally Lipschitz-continuous in $\eta$,
uniformly in $t\in[0,\tau]$. For simplicity, we will omit the explicit notation of time dependence in $F[\cdot]$ and $G[\cdot]$.
For any $\eta^1, \eta^2\in B_R^d$, where $B_R$ is the ball of radius $R>0$ in $\R^d$, we have
\[
   \left| \frac{F[\eta^1]}{G[\eta^1]} - \frac{F[\eta^2]}{G[\eta^2]} \right| \leq
      \left| \frac{F[\eta^1]-F[\eta^2]}{G[\eta^1]}\right| + |F[\eta^2]| \left| \frac{G[\eta^1]-G[\eta^2]}{G[\eta^1]G[\eta^2]} \right|.
\]
Since for $i=1,2$, $|\eta^i-\eta_{t-\tau}| \leq R + \max_{s\in[-\tau,0]} \Norm{\eta_s}_{L^\infty(\Omega_0)} < +\infty$, and due to the monotonicity
of the influence function $\psi$, we have
\[
   \psi(\eta^i - \eta_{t-\tau}(y)) \geq \psi(R + \max_{s\in[-\tau,0]} \Norm{\eta_s}_{L^\infty(\Omega_0)}) =: \psi_R > 0.
\]
Therefore, $G[\eta^i] \geq \psi_R$.
Due to the assumption $0 <\psi \leq 1$, we have the bound
\[
   |F[\eta^2]| \leq  \max_{s\in[-\tau,0]} \Norm{v_s}_{L^\infty(\Omega_0)} < +\infty.
\]
Moreover,
\[
   |F[\eta^1]-F[\eta^2]| \leq L_\psi |\eta^1-\eta^2| \max_{s\in[-\tau,0]} \Norm{v_s}_{L^\infty(\Omega_0)},
\]
where $L_\psi$ is the Lipschitz constant of the influence function $\psi$; a similar estimate obviously holds for $|G[\eta^1]-G[\eta^2]|$.
Note that the Lipschitz continuity of $\psi$ follows from Assumption \ref{ass:psi}.
Putting the above estimates together, we conclude that there exists a constant $C_R$, independent of $t\in(0,\tau)$,
such that
\[
   \left| \frac{F[\eta^1]}{G[\eta^1]} - \frac{F[\eta^2]}{G[\eta^2]} \right| \leq C_R |\eta^1-\eta^2|.
\]
Consequently, the Cauchy-Lipschitz theorem provides the existence of a unique solution $(\eta_t(x), v_t(x))$
of the system \eqref{ODEs1}--\eqref{ODEs2} 
on the time interval $(0,T_x)$ for some $0 < T_x < \tau$, for all $x\in\Omega_0$.

Next, still for a fixed but arbitrary $x\in\Omega_0$, we derive an a-priori bound on $(\eta_t(x), v_t(x))$.
We multiply \eqref{ODEs2} by $v_t(x)$,
\(
   \frac12 \frac{\d |v_t(x)|^2}{\d t} &=& \frac{F_t[\eta_t(x)]}{G_t[\eta_t(x)]}\cdot v_t(x) - |v_t(x)|^2  \nonumber \\
      &\leq& |v_t| \max_{s\in[-\tau,0]} \Norm{v_s}_{L^\infty(\Omega_0)} - |v_t(x)|^2,  \label{est_v}
\)
where we used the trivial inequality
\[
   \left| \frac{F_t[\eta_t(x)]}{G_t[\eta_t(x)]} \right| \leq \Norm{v_{t-\tau}}_{L^\infty(\Omega_0)}
      \leq \max_{s\in[-\tau,0]} \Norm{v_s}_{L^\infty(\Omega_0)}.
\]
Clearly, \eqref{est_v} implies that
\[
   |v_t(x)| \leq \max_{s\in[-\tau,0]} \Norm{v_s}_{L^\infty(\Omega_0)} \qquad\mbox{for } t\in(0,T_x),
\]
and from \eqref{ODEs1} we have
\[
   |\eta_t(x)| &\leq& |x| + T_x \max_{s\in[-\tau,0]} \Norm{v_s}_{L^\infty(\Omega_0)} \\
    &\leq& \max_{y\in\Omega_0} |y| + \tau \max_{s\in[-\tau,0]} \Norm{v_s}_{L^\infty(\Omega_0)}\qquad\mbox{for } t\in(0,T_x).
\]
Consequently, the solution constructed above is uniformly bounded on $[0,T_x]$
and can be extended to the whole interval $[0,\tau]$.
Since the bound is independent of $x\in\Omega_0$, we have
\[
   \max_{t\in[0,\tau]} \Norm{v_t}_{L^\infty(\Omega_0)} \leq \max_{s\in[-\tau,0]} \Norm{v_s}_{L^\infty(\Omega_0)},
\]
and we can re-iterate the above procedure to construct a family (parametrized by $x\in\Omega_0$)
of unique, global in time solutions of the system \eqref{ODEs1}--\eqref{ODEs2},
satisfying \eqref{global_v}.

Finally, we show that the continuity of the initial datum is propagated in time.
Let us fix $\eps>0$ and $x, z\in\Omega_0$ such that $|x-z|\leq\eps$.
Then $|\eta_0(x)-\eta_0(z)| = |x-z| \leq\eps$ and there exists a $\delta>0$
such that $|v_0(x)-v_0(z)| \leq \delta$.
An easy calculation gives
\(   \label{est1}
   |\eta_t(x)-\eta_t(z)| \leq |x-z| + \int_0^t |v_s(x)-v_s(z)| \d s,
\)
and, for almost all $t\in[0,T]$,
\(   \label{est2}
    \frac{\d}{\d t} |v_t(x)-v_t(z)| \leq  \left| \frac{F_t[\eta_t(x)]}{G_t[\eta_t(x)]} - \frac{F_t[\eta_t(z)]}{G_t[\eta_t](z)} \right|  - |v_t(x)-v_t(z)|.
\)
By a similar procedure as above we conclude that there exists a constant $C_1>0$, depending on $L_\psi$, the initial datum and $T>0$,
such that
\[
   \left| \frac{F_t[\eta_t(x)]}{G_t[\eta_t(x)]} - \frac{F_t[\eta_t(z)]}{G_t[\eta_t](z)} \right| \leq C_1 |\eta_t(x)-\eta_t(z)|
      \qquad \mbox{for all } t\in[0,T].
\]
Inserting into \eqref{est2} gives
\[
    \frac{\d}{\d t} |v_t(x)-v_t(z)| \leq C_1 |x-z| + C_1 \int_0^t |v_s(x)-v_s(z)| \d s   - |v_t(x)-v_t(z)|
\]
for almost all $t\in[0,T]$. Lemma \ref{lem:growth} implies then
\[
   |v_t(x)-v_t(z)| \leq \left( |v_0(x)-v_0(z)| + |x-z| \right) e^{\sqrt{C_1} t} \leq (\delta + \eps) e^{\sqrt{C_1}T}.
\]
Continuity in time is proved analogously, using again the Lipschitz continuity of $\psi$.
Consequently, $v_t \in \mc([0,T] \times \Omega_0)$ for any $T>0$.
Continuous differentiability in time and continuity in space of $\eta_t$ on $[0,T] \times \Omega_0$ follows directly from \eqref{est1}.
\end{proof}

%
%
\section{Large time behavior - proof of Theorem \ref{thm:flocking}} \label{sec:flocking}
In this Section we derive asymptotic estimates describing the large-time
behavior of the solutions $(\eta_t, v_t)$ of the system \eqref{Lagr1}, \eqref{LagrIC}, which can be constructed by Theorem \ref{thm:existence}.
For this whole Section we adopt the assumptions of Theorem \ref{thm:flocking} and Assumption \ref{ass:psi} with $\ell = 0$,
in particular, the validity of the formulae \eqref{R_V} and \eqref{iii}.

\begin{lemma}\label{lem_spt}
Let $(\eta_t, v_t)\in\mc^1([0,\infty); \mc(\overline{\Omega_0}))$ be a solution of the system \eqref{Lagr1}. Then we have
\bq\label{est_spt}
   \max_{x \in \overline\Omega_0} |v_t(x)| \leq R_V \qquad \mbox{for } t\geq 0,
\eq
with $R_V$ defined in \eqref{R_V}.
\end{lemma}

\begin{proof}
We fix an $\e > 0$, set $R_V^\e := R_V + \e$ and
\[
   \mathcal{A}^\e := \lt\{t > 0:\, \max_{x \in \overline \Om_0}|v_s(x)| < R^\e_V \quad \mbox{for } s \in [0,t)  \rt\}.
\]
Then, by assumption \eqref{R_V} and the continuity of the solution, we have
$\mathcal{A}^\e \neq \emptyset$ and $T^\e_* := \sup \mathcal{A}^\e > 0$.
For a contradiction, let us assume that $T^\e_* < +\infty$. Then we have
\(  \label{limit}
   \lim_{t \to T^\e_* -}\, \max_{x \in \overline \Om_0} |v_t(x)| = R_V^\e.
\)
On the other hand, for $t < T^\e_*$ we calculate
$$\begin{aligned}
\frac12 \frac{\d}{\d t}|v_t(x)|^2 &\leq
\frac{\int_{\Om_0} \psi(\eta_t(x) - \eta_{t-\tau}(y))|v_{t-\tau}(y)|\rho_0(y)\d y}{\int_{\Om_0} \psi(\eta_t(x) - \eta_{t-\tau}(y))\rho_0(y)\d y} |v_t(x)|- |v_t(x)|^2\cr
&\leq \max_{y \in \overline \Om_0}|v_{t-\tau}(y)||v_t(x)| - |v_t(x)|^2\cr
&\leq R^\e_V|v_t(x)| - |v_t(x)|^2.
\end{aligned}$$
Consequently,
\[
   \frac{\d}{\d t}|v_t(x)| \leq R^\e_V - |v_t(x)| \qquad \mbox{for almost all } t\in (0,T^\e_*),
\]
which further implies
\[
   |v_t(x)| \leq (|v_0(x)| - R^\e_V)e^{-t} + R^\e_V \qquad \mbox{for } t\in (0,T^\e_*).
\]
Thus, we have
\[
   \lim_{t \to T^\e_* -}\, \max_{x \in \overline \Om_0} |v_t(x)| \leq
      (\max_{x \in \overline \Om_0}|v_0(x)| - R^\e_V)e^{-T^\e_*} + R^\e_V \leq  - \eps e^{-T^\e_*} + R_V^\e,
\]
which is a contradiction to \eqref{limit}.
Hence we have $T^\e_* = +\infty$, and by taking the limit $\e \to 0$ we conclude \eqref{est_spt}.
\end{proof}

For $t\geq 0$ we define the quantities
\(
   X(t) &:=& d_X(0) + \int_0^t d_V(s) \d s, \label{X} \\
   V(t) &:=& d_V(0) e^{-t} + \int_0^t [1 - \psi(X(s-\tau) + R_V \tau)] d_V(s-\tau) e^{s-t} \d s, \label{V}
\)
with $d_X$ and $d_V$ defined in \eqref{dXdV} and $R_V$ defined in \eqref{R_V}.
Moreover, for $t\in[-\tau,0]$ we set
$$X(t):= d_X(t),\qquad V(t):=d_V(t),$$
so that both $X(t)$ and $V(t)$ are continuous on $[-\tau,\infty)$.

\begin{lemma}\label{main_prop}
Let $(\eta_t, v_t)\in\mc^1([0,\infty); \mc(\overline{\Omega_0}))$ be a solution of the system \eqref{Lagr1}. Then, for all $t>0$ we have
\(
\frac{\d}{\d t} X(t) &=& d_V(t),  \nonumber \\
\frac{\d}{\d t} V(t) &\leq& -V(t) + [1 - \psi(X(t-\tau) + R_V \tau)]V(t-\tau).  \label{ddtV}
\)
\end{lemma}

\begin{proof}
While the first claim follows directly from \eqref{X}, for the second we take the time derivative in \eqref{V},
\[
   \frac{\d}{\d t} V(t) = -V(t) + [1 - \psi(X(t-\tau) + R_V \tau)] \frac{\d}{\d t} X(t-\tau).
\]
Then, due to the inequality $\frac{\d}{\d t} X(t) \leq d_V(t)$ for all $t > -\tau$ and
the assumption $0 \leq \psi(s) \leq 1$ for $s\in [0,+\infty)$, we need to prove that
\(  \label{dVV}
   d_V(t) \leq V(t) \qquad\mbox{for } t\geq -\tau.
\)
For $x$, $z\in\Omega_0$ and $t\geq 0$ set
\[
   \phi_t(x,z) := \frac{\psi(\eta_t(x) - \eta_{t-\tau}(z))}{\int_{\Omega_0} \psi(\eta_t(x) - \eta_{t-\tau}(y))\rho_0(y)\d y}.
\]
Then, for any $x$, $y\in\Omega_0$ and $t>0$, we calculate
$$\begin{aligned}
\frac12\frac{\d}{\d t}|v_t(x) - v_t(y)|^2 &= \lt(v_t(x) - v_t(y) \rt) \cdot \lt(\pa_t v_t(x) - \pa_t v_t(y) \rt)\cr
&= \lt(v_t(x) - v_t(y) \rt) \cdot \int_{\Om_0} \lt(\phi_t(x,z) - \phi_t(y,z) \rt) v_{t-\tau}(z)\rho_0(z)\d z\cr
&\quad - |v_t(x) - v_t(y)|^2\cr
&\leq |v_t(x) - v_t(y)| \lt|\int_{\Om_0} \lt(\phi_t(x,z) - \phi_t(y,z)\rt) v_{t-\tau}(z)\rho_0(z)\d z \rt| \cr
&\quad - |v_t(x) - v_t(y)|^2.
\end{aligned}$$
We denote 
\[
  \wt \phi_t(x,y;z) := \min \lt\{ \phi_t(x,z), \phi_t(y,z)\rt\} \qquad \mbox{and} \qquad
  \Phi_t(x,y) := \int_{\Omega_0} \wt\phi_t(x,y;z)\rho_0(z)\d z.
\]
Then, by definition, $0 \leq \psi_t(x,y) \leq 1$ for all $x$, $y\in\Omega_0$ and $t\geq 0$.
Consequently,
\[
   \frac{\phi_t(x,z) - \wt\phi_t(x,y;z)}{1 - \Phi_t(x,y)} \geq 0 \qquad \mbox{and} \qquad 
   \int_{\Omega_0} \lt(\frac{\phi_t(x,z) - \wt\phi_t(x,y;z)}{1 - \Phi_t(x,y)} \rt) \rho_0(z)\d z = 1,
\]
which further implies
\[
   \int_{\Omega_0} \lt(\frac{\phi_t(x,z) - \wt\phi_t(x,y;z)}{1 - \Phi_t(x,y)} \rt) v_{t-\tau}(z)\rho_0(z)\d z \in
   \mbox{conv} \lt\{v_{t-\tau}(z),\, z \in \Omega_0 \rt\},
\]
where conv $\mathcal{S}$ denotes the convex hull of the set $\mathcal{S}$.
Therefore,
$$\begin{aligned}
&\lt|\int_{\Omega_0} \lt(\phi_t(x,z) - \phi_t(y,z)\rt) v_{t-\tau}(z)\rho_0(z)\d z \rt| \cr
&\quad \leq \lt(1 - \Phi_t(x,y) \rt) \Bigg|\int_{\Omega_0} \lt(\frac{\phi_t(x,z) - \wt\phi_t(x,y;z)}{1 - \Phi_t(x,y)} \rt) v_{t-\tau}(z)\rho_0(z)\d z \cr
&\qquad \qquad \qquad \qquad  \qquad  \qquad - \int_{\Omega_0}
   \lt(\frac{\phi_t(y,z) - \wt\phi_t(x,y;z)}{1 - \Phi_t(x,y)} \rt) v_{t-\tau}(z)\rho_0(z)\d z\Bigg|\cr
&\quad \leq \lt(1 - \Phi_t(x,y) \rt) d_V(t-\tau).
\end{aligned}$$
On the other hand, it follows from \eqref{eta_flow} and Lemma \ref{lem_spt} that
\[
   |\eta_t(x)-\eta_{t-\tau}(z)| = \lt|\eta_{t-\tau}(x) - \eta_{t-\tau}(z) - \int_{t-\tau}^t \frac{\d}{\d s}\eta_s(x)\d s \rt|
   \leq \lt|\eta_{t-\tau}(x) - \eta_{t-\tau}(z)\rt| + R_V\tau,
\]
and with \eqref{dXdV} we have
\[  \label{ineq_eta_xz}
   |\eta_t(x)-\eta_{t-\tau}(z)| \leq d_X(t-\tau) + R_V \tau.
\]
Due to the assumption $0 \leq \psi(s) \leq 1$ for $s\in [0,+\infty)$, we have
\[
  \int_{\Omega_0} \psi(\eta_t(x) - \eta_{t-\tau}(y))\rho_0(y)\d y \leq 1,
\]
and since $\psi$ is a nonincreasing function,
\[
   \phi_t(x,z) \geq \psi(\eta_t(x) - \eta_{t-\tau}(z)) \geq \psi(d_X(t-\tau) + R_V\tau)
\]
for all $x, z\in\Omega_0$.
This implies
\[
   \Phi_t(x,y) \geq \psi(d_X(t-\tau) + R_V\tau)
\]
for all $x, y\in\Omega_0$.
Combining the above estimates, we arrive at
\[
   \frac12\frac{\d}{\d t}|v_t(x) - v_t(y)|^2 \leq \bigl([1 - \psi(d_X(t-\tau) + R_V \tau)]d_V(t-\tau) - |v_t(x) - v_t(y)| \bigr) |v_t(x) - v_t(y)|.
\]
We divide by $|v_t(x) - v_t(y)|$ and integrate in time, which gives
\[
   |v_t(x) - v_t(y)| \leq |v_0(x) - v_0(y)|e^{-t} + \int_0^t \left[(1 - \psi(d_X(s-\tau) + R_V \tau)\right]d_V(s-\tau) e^{s-t} \d s,
\]
and taking the maximum over $x$, $y\in\overline\Omega_0$,
\[
   d_V(t) \leq d_V(0) e^{-t} + \int_0^t \left[(1 - \psi(d_X(s-\tau) + R_V \tau)\right]d_V(s-\tau) e^{s-t} \d s.
\]
Since, as can be easily proven, $d_X(t) \leq X(t)$,
the monotonicity property of the influence function $\psi$ finally implies
\[
   d_V(t) \leq d_V(0) e^{-t} + \int_0^t \left[(1 - \psi(X(s-\tau) + R_V \tau)\right]d_V(s-\tau) e^{s-t} \d s = V(t).
\]
\end{proof}

We next recall a Gronwall-type estimate for time-delayed differential inequalities whose proof can be found in \cite[Lemma 2.4]{Choi-Haskovec}.
\begin{lemma}\label{lem_gron}
Let $u$ be a nonnegative, continuous and piecewise $\mc^1$-function satisfying,
for some constant $0 < a < 1$, the differential inequality
\[
   \frac{\d}{\d t} u(t) \leq (1-a) u(t -\tau) - u(t) \qquad\mbox{for almost all } t>0.
\]
Then there exists a constant $0 < C < 1$ satisfying the equation
\[
  1 - C = (1-a)e^{C\tau},
\]
such that the estimate holds
\[
   u(t) \leq \left( \max_{s\in[-\tau,0]} u(s) \right) e^{-Ct} \qquad \mbox{for all } t \geq 0.
\]
\end{lemma}

We are now ready to prove the main result of this section.

\begin{lemma}\label{prop_lt}
Let $(\eta_t, v_t)\in\mc^1([0,\infty); \mc(\overline{\Omega_0}))$ be a solution of the system \eqref{Lagr1}. Assume that \eqref{R_V} and \eqref{iii} are verified.
Then we have
\[
d_V(t) \leq \left( \max_{s\in[-\tau,0]} d_V(s) \right) e^{-C_1t} \quad \mbox{for all } t>0,
   \quad \mbox{and} \quad \sup_{t>0} d_X(t) < C_2,
\]
where $C_1$, $C_2$ are positive constants independent of $t$.
\end{lemma}

\begin{proof}
We introduce the following Lyapunov functional for $t \in (0,T]$,
\[
   \Lyap(t) := V(t) + \int_{X(- \tau) + R_v \tau}^{X(t - \tau) + R_v \tau} \psi(s)\,\d s + \int_{t-\tau}^t V(s)\,\d s.
\]
Using Lemma \ref{main_prop}, we obtain that for almost all $t \in (0,T]$,
$$
\begin{aligned}
   \tot{}{t}\Lyap(t) &= \tot{}{t} V(t) + \psi(X(t - \tau) + R_v \tau) \tot{}{t} X(t - \tau) + V(t) - V(t - \tau)\cr
   &\leq - V(t) + \lt[1 - \psi(X(t-\tau) + R_v \tau)\rt] V(t - \tau) \cr 
   &\quad + \psi(X(t - \tau) + R_v \tau) d_V(t-  \tau) + V(t) - V(t - \tau)\cr
   &= 0,
\end{aligned}
$$
where we used the inequality $d_V(t) \leq V(t)$ for $t\geq -\tau$, \eqref{dVV} from the proof of Lemma \ref{main_prop}.
Integrating over the time interval $(0,t)$ yields
\(  \label{zwischenstep}
   V(t) + \int_{X(- \tau) + R_v \tau}^{X(t - \tau) + R_v \tau} \psi(s)\,\d s + \int_{t-\tau}^t V(s)\,\d s \leq V(0)  + \int_{-\tau}^0 V(s)\,\d s.
\)
On the other hand, assumption \eqref{iii} implies that there exists a $d_* > 0$ such that
\[
   d_V(0)  + \int_{-\tau}^0 d_V(s)\,\d s = \int_{d_X(-\tau) + R_v \tau}^{d_*}\psi(s)\,\d s.
\]
Since, by definition, $V(t) = d_V(t)$ for $t\in[-\tau,0]$, we have
\[
   d_V(0)  + \int_{-\tau}^0 d_V(s)\,\d s = V(0)  + \int_{-\tau}^0 V(s)\,\d s.
\]
This together with \eqref{zwischenstep} implies
\[
   \int_{X(- \tau) + R_v \tau}^{X(t - \tau) + R_v \tau} \psi(s)\,\d s \leq \int_{d_X(-\tau) + R_v \tau}^{d_*}\psi(s)\,\d s,
\]
and, since $X(- \tau) = d_X(-\tau)$,
\[
    0  \leq \int_{X(t -\tau) + R_v \tau}^{d_*} \psi(s)\,\d s.
\]
With the inequality $d_X(t) \leq X(t)$ for $t\geq -\tau$, this implies
\[
   d_X(t -\tau) + R_v \tau \leq X(t -\tau) + R_v \tau \leq d_* \quad \mbox{for }  t > 0.
\]
Using this in \eqref{ddtV}, we arrive at
\[
   \tot{}{t} V(t) \leq -V(t) + (1 - \psi_*) V(t - \tau)
\]
for all $t>0$, where $\psi_* := \psi(d_*)$.
We finally apply Lemma \ref{lem_gron} and the inequality \eqref{dVV} to complete the proof.
\end{proof}

\begin{remark}
In the above proof we proceeded along the lines of \cite{Choi-Haskovec},
where a similar statement has been proved for the discrete setting.
However, in our setting the proof becomes slightly more involved.
Indeed, in the discrete setting
the time axis can be divided into an at most countable
system of disjoint intervals $[t_k, t_{k+1})$
such that the velocity diameter $d_V$ is realized
by a fixed pair of particles on this time interval,
and one can calculate the time derivative of $d_V$ there.
This is obviously not possible in the continuum setting.
Consequently, we had to introduce the functions $X$ and $V$ in \eqref{X}--\eqref{V}
and estimate their time derivatives in Lemma \ref{main_prop}.
This is the main difference with respect to the approach taken in \cite{Choi-Haskovec}.
\end{remark}

%
%
\section{Existence of solutions for the Eulerian system - proof of Theorem \ref{thm:Eulerian}}\label{sec:classical}
In this section we prove the global-in-time existence and uniqueness of classical solutions of the system \eqref{Eul1}--\eqref{Eul2}. We first shortly establish the existence of local-in-time solutions in Lemma \ref{lem_local} and then derive suitable a-priori estimates that allow us to establish a global-in-time result.

\begin{lemma}\label{lem_local}
Let Assumption \ref{ass:psi} be verified with some $\ell > \frac{d}2+1$. Moreover, we assume that the initial datum satisfying the regularity: 
\[
(\bar\rho_s, \bar u_s) \in \mc([-\tau,0];H^\ell(\Omega_s)) \times \mc([-\tau,0];H^{\ell+1}(\Omega_s)).
\]
Then, there exists a $T > 0$ such that the system \eqref{Lagr1} has a unique classical solution $(\eta_t, v_t) \in \mc^1([-\tau,T];H^{\ell+1}(\Om_0)) \times \mc([-\tau,T];H^{\ell+1}(\Om_0))$.
\end{lemma}

\begin{proof}
Even though we are dealing with the effect of time delay in the alignment force, the existence and uniqueness of local-in-time classical solutions can be obtained by using a similar argument as in \cite[Appendix A]{CCZ}, in which the one-dimensional pressureless Euler-Poisson equations are considered. Thus we skip it here.
\end{proof}

For the global regularity, we need to use the estimate of time behavior studied in Section \ref{sec:flocking} which plays a important role in constructing the global-in-time classical solutions. Similar idea are used in \cite{CK16, HKK15} to prevent the formation of finite-time singularities in pressureless Eulerian dynamics. We derive uniform a priori estimates for $\|v_t\|_{H^{\ell+1}}$ to the equation $\eqref{Lagr1}_2$. For this, we recall several Sobolev inequalities which will be used in the rest of this paper.





\begin{lemma}\label{lem_sob} Let $k \geq 1$.
\begin{itemize}
\item[(i)] For any pair of functions $f,g \in H^k \cap L^\infty$, we obtain
\[
\|\nabla^k (fg)\|_{L^\infty} \ls \|f\|_{L^\infty}\|\nabla^k g\|_{L^\infty} + \|\nabla^k f\|_{L^\infty}\|g\|_{L^\infty}.
\]
Here $f \ls g$ represents that there exists a positive constant $C>0$ such that $f \leq Cg$.

Furthermore, if $\nabla f \in L^\infty$, we have
\[
\|\nabla^k (fg) - f \nabla^k g\|_{L^2} \ls \|\nabla f\|_{L^\infty}\|\nabla^{k-1}g\|_{L^2} + \|\nabla^k f\|_{L^2}\|g\|_{L^\infty}.
\]
\item[(ii)] For $f \in H^{[d/2]+1}$, we have
\[
\|f\|_{L^\infty} \ls \|\nabla f\|_{H^{[d/2]}}.
\]
\item[(iii)] For $f \in H^k \cap L^\infty$, let $p \in [1,\infty]$, and $h \in \mc^k(B(0,\|f\|_{L^\infty}))$ where $B(0,R)$ denotes the ball of radius $R>0$ centered at the origin in $\R^d$. Then there exists a positive constant $C = C(k,p,h)$ such that
\[
\|\nabla^k h(f)\|_{L^p} \leq C(1 + \|f\|_{L^\infty})^{k-1}\|\nabla^k f\|_{L^p}.
\]
\end{itemize}
\end{lemma}
For notational simplicity, we set
\[  \label{ppsi}
   \ppsi(x,y) := \psi(\eta_t(x) - \eta_{t-\tau}(y)).
\]
In the lemma below, we provide the $H^k$-estimate of the influence function $\ppsi$ by directly using Lemma \ref{lem_sob}.
\begin{lemma}\label{lem_useful0}Let $\eta_t \in \mc([-\tau,T];H^{\ell+1}(\Omega_0))$ for some $T>0$ and $\ell\in\N$. Then, for $1 \leq k \leq \ell + 1$, we have
\[
\|\nabla^k_x \ppsi(\cdot,y)\|_{L^2} \leq C(1 + d_X(t-\tau) + R_V \tau)^{k-1}\|\nabla^k_x \eta_t(\cdot)\|_{L^2}.
\]
In particular, we have
\[
\|\nabla^k_x \ppsi(\cdot,\cdot)\|_{L^2 \times L^\infty} \leq C(1 + d_X(t-\tau) + R_V \tau)^{k-1}\|\nabla^k_x \eta_t(\cdot)\|_{L^2}.
\]
\end{lemma}

\begin{remark}\label{rmk_useful0}Due to the smoothness of influence function $\psi$, we can easily get
\[
\|\nabla_x \ppsi(\cdot,\cdot)\|_{L^\infty \times L^\infty} \leq C\|\nabla_x \eta_t(\cdot)\|_{L^\infty}.
\]

\end{remark}

\begin{lemma}\label{lem_useful}
Let $\ell > d/2 + 1$ and $T > 0$. Suppose that the assumptions given in Theorem \ref{thm_main} hold. Then we have
\[
\int_{\Om_0} \ppsi(x,y) \rho_0(y)\,dy \geq \psi(C_2 + R_V\tau) =: \psi_m > 0 \quad \mbox{for all } x\in\Omega_0 \mbox{ and } t \in [0,T], 
\]
with $C_2 > 0$ given in Lemma \ref{prop_lt}. 

Furthermore, if there exists a positive constant $M > 0$ such that $\|v\|_{L^\infty((-\tau,T); H^{\ell + 1}(\Omega))} \leq M$, we have
\[
\|\nabla^k \eta_t\|_{L^2(\Omega_0)} \leq C(1 + M t) \quad \mbox{for } 1 \leq k \leq \ell+1 \mbox{ and } t \in [0,T], 
\]
for some $C>0$ independent of $t$.
\end{lemma}
\begin{proof} It follows from the monotonicity of the influence function $\psi$, Lemma \ref{prop_lt}, and the
inequality \eqref{ineq_eta_xz} given in the proof of Lemma \ref{main_prop} that
\[
   |\eta_t(x)-\eta_{t-\tau}(y)| \leq d_X(t-\tau) + R_V \tau.
\]
Thus we obtain
\[
\int_{\Om_0} \ppsi(x,y) \rho_0(y)\,dy \geq \psi(C_2 + R_V\tau)\int_{\Om_0} \rho_0(y)\,dy = \psi_m.
\]
We now assume that $\|v\|_{L^\infty((-\tau,T); H^{\ell + 1}(\Omega))} \leq M$ for some $M > 0$. Taking the $k$-th derivative of $\eqref{Lagr1}_1$ yields
\[
   \nabla^k \eta_t = \delta_{k,1} \mathbb{I} + \int_0^t \nabla^k v_s\d s \quad \mbox{for} \quad 1 \leq k \leq \ell + 1,
\]
where $\mathbb{I}$ is the identity matrix.
This yields
\[
   \|\nabla^k \eta_t\|_{L^2(\Omega_0)} \leq C \left(1 + \int_0^t \|\nabla^k v_s\|_{L^2(\Omega_0)}\d s\right)  \leq C(1 +  Mt).
\]
\end{proof}

\begin{remark}\label{rmk_gd} It follows from Lemmas \ref{lem_useful0} and \ref{lem_useful}, Remark \ref{rmk_useful0}, and the Sobolev embedding $H^{\ell-1}(\Omega_0) \hookrightarrow L^\infty(\Omega_0)$
for $\ell > d/2 + 1$ that
\[
\|\nabla^k_x \ppsi(\cdot,\cdot)\|_{L^2 \times L^\infty} + \|\nabla_x \ppsi(\cdot,\cdot)\|_{L^\infty \times L^\infty} \leq C(1+Mt),
\]
for $1 \leq k \leq \ell+1$.
\end{remark}


We are now in a position to provide the uniform a priori estimate of $\|v_t\|_{H^{\ell+1}}$ in the lemma below.

\begin{lemma}\label{prop_apriori}
Let $\ell > d/2 + 1$ and $T > 0$.
Suppose that the assumptions given in Theorem \ref{thm_main} hold. Let $M > 0$ be any positive constant. Then if $\|v\|_{L^\infty((-\tau,T); H^{\ell + 1}(\Omega_0))} \leq M$, we have
\[
   \|v_t\|_{L^\infty((0,T);H^{\ell+1}(\Omega_0))} \leq C_0\|v_s\|_{L^\infty((-\tau,0);H^{\ell+1}(\Omega_0))},
\]
where $C_0 >0$ is a constant independent of $T$.
\end{lemma}
\begin{remark} The constant $M>0$ appeared in Lemma \ref{prop_apriori} does not need to be small.
\end{remark}
\begin{proof}[Proof of Lemma \ref{prop_apriori}]
We start by estimating the $L^2(\Omega_0)$-norm of $v_t$ for a fixed $t\in (0,T)$.
We calculate
$$\begin{aligned}
\frac12\frac{\d}{\d t} \int_{\Omega_0} |v_t|^2\d x &=
\int_{\Omega_0} v_t \cdot \lt(\frac{\int_{\Om_0} \ppsi(x,y) v_{t-\tau}(y)\rho_0(y)\d y}{\int_{\Om_0} \ppsi(x,y)\rho_0(y)\d y}  - v_t\rt)\d x\cr
&\leq \|v_t\|_{L^1(\Omega_0)}\|v_{t-\tau}\|_{L^\infty(\Omega_0)} - \|v_t\|_{L^2(\Omega_0)}^2\cr
&\leq {|\Omega_0|}^{1/2}\|v_t\|_{L^2(\Omega_0)}\|v_s\|_{L^\infty((-\tau,0); L^\infty(\Omega_0))} - \|v_t\|_{L^2(\Omega_0)}^2\cr
&\leq -\frac12\|v_t\|_{L^2(\Omega_0)}^2 + C\|v_s\|_{L^\infty((-\tau,0);H^{\ell+1}(\Omega_0))}^2,
\end{aligned}$$
where we used Lemma \ref{lem_spt} and the Cauchy-Schwartz inequality,
and $C > 0$ only depends on $|\Omega_0|$ and the space dimension $d\in\N$.
An application of the Gronwall lemma gives then
\[
   \sup_{0 \leq t \leq T}\|v_t\|_{L^2(\Omega_0)}^2 \leq \|v_0\|_{L^2(\Omega_0)}^2
      + C\|v_s\|_{L^\infty((-\tau,0);H^{\ell+1}(\Omega_0))}^2 \leq C\|v_s\|_{L^\infty((-\tau,0);H^{\ell+1}(\Omega_0))}^2.
\]

Next, we estimate the $H^1(\Omega_0)$-norm of $v_t$,
$$\begin{aligned}
   \frac12\frac{\d}{\d t}\int_{\Om_0}|\nabla v_t|^2\d x &=
   \int_{\Omega_0} \nabla v_t \cdot \nabla
      \left( \frac{\int_{\Omega_0} \ppsi(x,y) v_{t-\tau}(y)\rho_0(y)\d y}{\int_{\Om_0} \ppsi(x,y)\rho_0(y)\d y} - v_t \right) \d x \\
   & =: - \|\nabla v_t\|_{L^2(\Omega_0)}^2 + I_1.
\end{aligned}$$
Note that 
$$\begin{aligned}
&\lt|\nabla_x \lt( \frac{\int_{\Omega_0} \ppsi(x,y) v_{t-\tau}(y)\rho_0(y)\d y}{\int_{\Om_0} \ppsi(x,y)\rho_0(y)\d y}\rt)\rt|\cr
&\quad = \lt|\frac{\iint_{\Omega_0 \times \Om_0} (\nabla_x \ppsi(x,y)) \ppsi(x,z)(v_{t-\tau}(y) - v_{t-\tau}(z))\rho_0(y)\rho_0(z)\d y\d z}{\lt(\int_{\Om_0} \ppsi(x,y)\rho_0(y)\d y\rt)^2}\rt|\cr
&\quad \leq \frac{1}{\psi_m^2}d_V(t-\tau)\left(\iint_{\Omega_0\times\Omega_0}|\nabla (\ppsi(x,y))|^2 \rho_0(y) \rho_0(z) \d y\d z\right)^{1/2} \\
&\qquad\qquad\qquad\times
\left(\iint_{\Omega_0\times\Omega_0}|\ppsi(x,y)|^2 \rho_0(y) \rho_0(z) \d y\d z\right)^{1/2}\cr
&\quad \leq \frac{1}{\psi_m^2}d_V(t-\tau)\lt(\int_{\Omega_0}|\nabla (\ppsi(x,y))|^2 \rho_0(y)\d y\rt)^{1/2},
\end{aligned}$$
due to Lemma \ref{lem_useful}, $\|\psi\|_{L^\infty} \leq 1$ and the normalization $\int_{\Omega_0} \rho_0(y) \d y = 1$.
The Cauchy-Schwarz inequality yields
\[
   |I_1| \leq \frac{1}{\psi_m^2} d_V(t-\tau)\|\nabla v_t\|_{L^2(\Omega_0)}
      \left(\iint_{\Omega_0 \times \Omega_0}|\nabla_x \ppsi(x,y)|^2 \rho_0(y)\d x\d y \right)^{1/2},
\]
and using further the estimate \eqref{statement2} on $d_V(t-\tau)$ together with Remark \ref{rmk_gd}, we obtain
$$\begin{aligned}
|I_1|
&\leq \frac{C}{\psi_m^2} \left( \max_{s \in [-\tau,0]}d_V(s) \right) e^{-C (t-\tau)}\|\nabla v_t\|_{L^2(\Omega_0)}(1 + \e_1t) |\Omega_0|^{1/2} \cr
&\leq C\|v_s\|_{L^\infty((-\tau,0)\times\Omega_0)}\|\nabla v_t\|_{L^2(\Omega_0)}\cr
&\leq \frac12\|\nabla v_t\|_{L^2(\Omega_0)}^2 + C\|v_s\|_{L^\infty((-\tau,0)\times\Omega_0)}^2,
\end{aligned}$$
where we used the elementary inequality $e^{-C t}(1 + M t) \leq C_{M}$ for all $t\geq 0$, with $C_{M} > 0$ independent of $T$.
Therefore, we have
\[
      \frac12\frac{\d}{\d t} \|\nabla v_t\|_{L^2(\Omega_0)}^2
         \leq - \frac12 \|\nabla v_0\|_{L^2(\Omega_0)}^2 + C\|v_s\|_{L^\infty((-\tau,0)\times\Omega_0)}^2,
\]
which implies
\[
   \sup_{0 \leq t \leq T}\|\nabla v_t\|_{L^2(\Omega_0)}^2
      \leq \|\nabla v_0\|_{L^2(\Omega_0)}^2 + C\|v_s\|_{L^\infty((-\tau,0); H^{\ell+1}(\Omega_0))}^2
      \leq C\|v_s\|_{L^\infty((-\tau,0); H^{\ell+1}(\Omega_0))}^2,
\]
where we used the embedding $H^{\ell+1}(\Omega_0) \hookrightarrow L^\infty(\Omega_0)$.

Finally, we derive the estimate of the $H^k$-norm of $v_t$ for general $k\in\N$.
We first notice that for $1 \leq k \leq \ell$,
$$\begin{aligned}
&\nabla^{k+1}_x \lt( \frac{\int_{\Om_0} \ppsi(x,y) v_{t-\tau}(y)\rho_0(y)\d y}{\int_{\Om_0} \ppsi(x,y)\rho_0(y)\d y}\rt)\cr
&= \nabla^{k}_x\lt( \frac{\iint_{\Om_0 \times \Om_0} \nabla_x (\ppsi(x,y)) \ppsi(x,z)(v_{t-\tau}(y) - v_{t-\tau}(z))\rho_0(y)\rho_0(z)\d y\d z}{\lt(\int_{\Om_0} \ppsi(x,y)\rho_0(y)\d y\rt)^2}\rt)\cr
&=: \sum_{1 \leq k' \leq k-1}\binom{k}{k'} \nabla^{k'}_x I_2(x) \nabla^{k-k'}_x I_3(x) (1 - \delta_{k,1}) + I_2(x) \nabla_x^k I_3(x) + \nabla_x^k I_2(x) I_3(x)\cr
&=: J_1(x) + J_2(x) + J_3(x),
\end{aligned}$$
where
$$\begin{aligned} 
I_2(x) &= \lt( \int_{\Om_0} \ppsi(x,y)\rho_0(y)\d y\rt)^{-2},\cr
I_3(x) &= \int_{\Om_0 \times \Om_0} (\nabla_x \ppsi(x,y)) \ppsi(x,z)(v_{t-\tau}(y) - v_{t-\tau}(z))\rho_0(y)\rho_0(z)\d y\d z.
\end{aligned}$$
Note that, due to Lemma \ref{lem_useful},
\[
   I_2(x) \leq \psi_m^{-2} \qquad \mbox{for } x \in \Omega_0.
\]
Furthermore, for $1 \leq k \leq \ell$,
$$\begin{aligned}
|\nabla_x^k I_2| &\ls \lt| \int_{\Om_0} \nabla^k (\ppsi (x,y))\rho_0(y)\d y \rt| \cr
&\quad + (1 - \delta_{k,1}) \sum_{ \substack{\alpha+ \beta = k \\ \alpha,\beta \geq 1}}\left| \int_{\Omega_0} \nabla_x^\alpha (\ppsi (x,y))\rho_0(y)\d y \right| \left| \int_{\Omega_0} \nabla_x^\beta (\ppsi (x,y))\rho_0(y)\d y \right|\cr
&=: I_2^1 + I_2^2,
\end{aligned}$$
where $L^2$-norm of $I_2^1$  can be easily estimated as
$$\begin{aligned}
\int_{\Om_0} |I_2^1|^2 \,\d x &\ls \int_{\Om_0} \lt| \int_{\Om_0} |\nabla^k (\ppsi (x,y))|\rho_0(y)\d y\rt|^2\,dx \cr
&\ls \lt(\int_{\Omega_0} \|\nabla_x^k \ppsi(\cdot,y)\|_{L^2(\Omega_0)}\rho_0(y)\d y\rt)^2\cr
&\ls \int_{\Omega_0} \|\nabla_x^k \ppsi(\cdot,y)\|_{L^2(\Omega_0)}^2\rho_0(y)\d y\cr
&\ls \|\nabla_x^k \ppsi(\cdot,\cdot)\|_{L^2 \times L^\infty}^2,
\end{aligned}$$
due to Minkowski integral inequality. For the estimate of $I_2^2$, we again use Minkowski integral inequality together with Moser-type inequalities to obtain
$$\begin{aligned}
  &\int_{\Omega_0} | I_2^2|^2\,\d x \cr
  &\quad \ls \sum_{ \substack{\alpha+ \beta = k \\ \alpha,\beta \geq 1}}\int_{\Om_0}\lt| \int_{\Om_0 \times \Om_0}|\nabla^\alpha_x(\ppsi(x,y))||\nabla^\beta_x(\ppsi(x,z))| \rho_0(y)\rho_0(z)\d y\d z\rt|^2\,dx \cr
&\quad \ls \sum_{ \substack{\alpha+ \beta = k \\ \alpha,\beta \geq 1}} \lt(\iint_{\Omega_0 \times \Omega_0} \bigl\lVert
     |\nabla_x^\alpha \ppsi(\cdot,y)| |\nabla_x^\beta \ppsi(\cdot,z)| \bigr\rVert_{L^2(\Omega_0)} \rho_0(y)\rho_0(z)\d y\d z\rt)^2\cr
  &\quad \ls \sum_{ \substack{\alpha+ \beta = k \\ \alpha,\beta \geq 1}} \iint_{\Omega_0 \times \Omega_0} \bigl\lVert
     |\nabla_x^\alpha \ppsi(\cdot,y)| |\nabla_x^\beta \ppsi(\cdot,z)| \bigr\rVert_{L^2(\Omega_0)}^2 \rho_0(y)\rho_0(z)\d y\d z\cr
&\quad \ls \sum_{ \substack{\alpha+ \beta = k \\ \alpha,\beta \geq 1}} \iint_{\Omega_0 \times \Omega_0}
      \| \nabla_x^\alpha \ppsi(\cdot,y)\|_{H^1(\Omega_0)}^2 \|\nabla_x^\beta \ppsi(\cdot,z)\|_{H^1(\Omega_0)}^2 \rho_0(y)\rho_0(z)\d y\d z\cr
&\quad \ls \|\nabla_x \ppsi(\cdot,\cdot)\|_{H^k \times L^\infty}^4.
\end{aligned}$$
Thus we obtain
\[
\|\nabla I_2\|_{H^{k-1}} \leq  C\lt( \|\nabla_x^k \ppsi(\cdot,\cdot)\|_{L^2 \times L^\infty}^2 + \|\nabla_x \ppsi(\cdot,\cdot)\|_{H^k \times L^\infty}^4\rt) \leq C(1 + M t)^4, 
\]
for $1 \leq k \leq \ell$. We also easily find from Remark \ref{rmk_gd} that
\[
|I_3(x)| \leq d_V(t-\tau)\|\nabla \ppsi(\cdot,\cdot)\|_{L^\infty \times L^\infty} \leq Cd_V(t-\tau)(1 + M t).
\] 
In a similar fashion as before, we get that for $1 \leq k \leq \ell$
$$\begin{aligned}
&\int_{\Om_0} |\nabla^k I_3|^2\,dx \cr
&\quad \leq d_V^2(t-\tau)\int_{\Om_0} \lt(\int_{\Om_0 \times \Om_0} |\nabla^k \lt(\nabla(\ppsi(x,y)) \ppsi(x,z) \rt) \rho_0(y)\rho_0(z)\,dydz\rt)^2 dx\cr
&\quad \leq d_V^2(t-\tau)\lt(\int_{\Om_0 \times \Om_0} \|\nabla^k \lt(\nabla(\ppsi(\cdot,y)) \ppsi(\cdot,z) \rt)\|_{L^2} \rho_0(y)\rho_0(z)\,dydz\rt)^2\cr
&\quad \leq Cd_V^2(t-\tau)\lt(\int_{\Om_0 \times \Om_0} \lt(\|\nabla (\ppsi(\cdot,y))\|_{L^\infty}\|\nabla^k (\ppsi(\cdot,z))\|_{L^2} \rt) \rho_0(y)\rho_0(z)\,dydz\rt)^2\cr
&\qquad + Cd_V^2(t-\tau)\lt(\int_{\Om_0} \|\nabla^{k+1} (\ppsi(\cdot,z))\|_{L^2}  \rho_0(z)\,dz\rt)^2\cr
&\quad \leq Cd_V^2(t-\tau)(1 + M t)^4,
\end{aligned}$$
and, subsequently, this implies
\[
\|\nabla I_3\|_{H^{k-1}} \leq Cd_V(t-\tau)(1 + M t)^2.
\]
Using the above estimates, we have that for $1 \leq k \leq \ell$
$$\begin{aligned}
\|J_1\|_{L^2} &\leq C(1 - \delta_{k,1})\lt(\int_{\Om_0}\lt|\sum_{1 \leq k' \leq k-1}\nabla^{k'}I_2(x) \nabla^{k-k'}I_3(x)\rt|dx \rt)^{1/2}\cr
&\leq C(1 - \delta_{k,1})\sum_{1 \leq k' \leq k-1} \| |\nabla^{k'} I_2| |\nabla^{k-k'}I_3|\|_{L^2}\cr
&\leq C(1 - \delta_{k,1})\sum_{1 \leq k' \leq k}\|\nabla^{k'}I_2\|_{H^1}\|\nabla^{k-k'}I_3\|_{H^1}\cr
&\leq C\|\nabla I_2\|_{H^{k-1}}\|\nabla I_3\|_{H^{k-1}} \leq Cd_V(t-\tau)(1 + M t)^4,\cr
\|J_2\|_{L^2} &\leq \|I_2\|_{L^\infty}\|\nabla^k I_3\|_{L^2} \leq Cd_V(t-\tau)(1 + M t)^4,\cr
\|J_3\|_{L^2} &\leq \|I_3\|_{L^\infty}\|\nabla^k I_2\|_{L^2} \leq Cd_V(t-\tau)(1 + M t)^6.
\end{aligned}$$
This yields
$$\begin{aligned}
&\lt\|\nabla^{k+1} \lt( \frac{\int_{\Om_0} \ppsi(x,y) v_{t-\tau}(y)\rho_0(y)\,dy}{\int_{\Om_0} \ppsi(x,y)\rho_0(y)\,dy}\rt) \rt\|_{L^2} \cr
&\qquad \leq Cd_V(t-\tau)(1 + M t)^8 \leq C\|v_s\|_{L^\infty(-\tau,0;H^{\ell+1})},
\end{aligned}$$
for $1 \leq k \leq \ell$. Finally, we have
$$\begin{aligned}
&\frac12\frac{d}{dt}\int_{\Om_0}|\nabla^{k+1} v_t|^2\,dx  \cr
&\quad = - \|\nabla^{k+1} v_t\|_{L^2}^2 + \int_{\Om_0} \nabla^{k+1} \lt( \frac{\int_{\Om_0} \ppsi(x,y) v_{t-\tau}(y)\rho_0(y)\,dy}{\int_{\Om_0} \ppsi(x,y)\rho_0(y)\,dy}\rt) \cdot \nabla^{k+1} v_t(x)\,dx\cr
&\quad \leq - \frac12\|\nabla^{k+1} v_t\|_{L^2}^2 + C\|v_s\|_{L^\infty(-\tau,0;H^{\ell+1})}^2,
\end{aligned}$$
for $1 \leq k \leq \ell$. Hence we conclude that
\[
\|\nabla^2 v_t\|_{H^{\ell-1}} \leq C\|v_s\|_{L^\infty(-\tau,0;H^{\ell+1})}.
\]
This completes the proof.
\end{proof}

\begin{proof}[Proof of Theorem \ref{thm_main}] The proof of global existence and uniqueness of classical solutions are easily obtained from Lemma \ref{lem_local} and Lemma \ref{prop_apriori}. Once we obtained the global-in-time classical solutions, the all computations in Section \ref{sec:flocking} are justified. This completes the proof.
\end{proof}

As mentioned in Section \ref{sec:main}, in order to back to the Eulerian variable from the Lagrangian formulation, we need to show that the characteristic flow $\eta_t$ defined in \eqref{eta_flow} is a diffeomorphism. Thus, in the rest of this section, we provide the estimate of det$\nabla \eta_t$. For this, we first need to estimate the exponential decay of $\nabla v_t$ in $L^2$-norm.

\begin{lemma}\label{lem_decay2}
Let Assumption \ref{ass:psi} be verified with some $\ell > \frac{d}2+1$, let \eqref{R_V} and \eqref{iii} hold. Then we have
\[
\|\nabla v_t\|_{L^2(\Omega_0)} \leq \|\nabla v_0\|_{L^2(\Omega_0)}e^{-t} + C\left( \max_{s \in [-\tau,0]}d_V(s) \right) e^{-at} \quad \mbox{for} \quad t \geq 0,
\]
for some $a > 0$.
\end{lemma}
\begin{proof}
Similarly as in Lemma \ref{prop_apriori}, we get
\[
\frac12\frac{\d}{\d t}\int_{\Om_0}|\nabla v_t|^2\d x \leq - \|\nabla v_t\|_{L^2(\Omega_0)}^2 + C\left( \max_{s \in [-\tau,0]}d_V(s) \right) e^{-C (t-\tau)}\|\nabla v_t\|_{L^2(\Omega_0)}(1 + t).
\]
This yields
\[
\frac{d}{dt} \|\nabla v_t\|_{L^2(\Omega_0)} \leq -\|\nabla v_t\|_{L^2(\Omega_0)} + C\left( \max_{s \in [-\tau,0]}d_V(s) \right)(1 + t)e^{-Ct}.
\]
Applying Gronwall's inequality gives
\[
\|\nabla v_t\|_{L^2(\Omega_0)} \leq \|\nabla v_0\|_{L^2(\Omega_0)}e^{-t} + C\left( \max_{s \in [-\tau,0]}d_V(s) \right) e^{-at} \quad \mbox{for} \quad t \geq 0,
\]
for some $a > 0$.
\end{proof}
We are ready to provide the details of the proof for Theorem \ref{thm:Eulerian}.
\begin{proof}[Proof of Theorem \ref{thm:Eulerian}]It suffices to show that det$(\nabla \eta_t) > 0$ for all $t \geq 0$. It follows from \eqref{Lagr1} that
\[
\nabla_x\eta_t(x) = \mathbb{I} + \int_0^t \nabla_x v_s(x)\,ds.
\]
Using Sobolev-Gagliardo-Nirenberg inequality together with the decay estimate in Lemma \ref{lem_decay2} yields
\[
\|\nabla_x v_t\|_{L^\infty} \leq C\|\nabla_x v_t\|_{H^s} \leq C\|\nabla_x v_t\|_{L^2}^{1-\beta}\|\nabla_x v_t\|_{H^{s+1}}^\beta \leq C\e e^{-bt},
\]
for some $b >0$ and $\beta \in (0,1)$, where $s > \frac d2$. Note that det$(\mathbb{I} + \e_0 A) = 1 + \e_0$tr$(A) + \mathcal{O}(\e_0^2)$ for some $\e_0> 0$ and 
\[
\lt|\int_0^t \nabla_x v_s(x)\,ds\rt| \leq C\e\int_0^t e^{-bs}\,ds \leq C\e.
\]
This provides
\[
\mbox{det}(\nabla_x\eta_t(x)) = \mbox{det}\lt(\mathbb{I} + \int_0^t \nabla_x v_s(x)\,ds \rt) \geq 1 - C\e - \mathcal{O}(\e^2).
\]
Hence, by choosing $\e > 0$ small enough, we conclude the desired result.
\end{proof}

%
%
\section{Critical threshold phenomenon in one dimension - Proof of Theorem \ref{thm_cri}} \label{sec:critical}
We conclude the paper by showing a critical threshold phenomenon in the spatially
one-dimensional version ($d=1$) of the system \eqref{Eul1}--\eqref{Eul2},
which can be rewritten as
\(
   \pa_t \rho_t + \pa_x(\rho_t u_t) &=& 0,  \label{crit1} \\
   \pa_t u_t + u_t \pa_x u_t  &=&  \frac{\int_\R \psi(x-y) u_{t-\tau}(y) \rho_{t-\tau}(y)\d y}{\int_\R \psi(x-y) \rho_{t-\tau}(y)\d y} - u_t.
   \label{crit2}
\)
We proceed in the spirit of \cite{CCTT, CCZ, TT} and set $w_t(x) := \pa_x u_t(x)$ and introduce the differential operator
$D_t := \pa_t + u_t \pa_x$.
Taking the formal $x$-derivative of \eqref{crit2}, we obtain
\[
   D_t w_t + w_t^2 = \partial_x \left( \frac{\int_\R \psi(x-y) u_{t-\tau}(y) \rho_{t-\tau}(y)\d y}{\int_\R \psi(x-y) \rho_{t-\tau}(y)\d y} \right)
     - w_t.
\]
Using Assumption \ref{ass:psi} with some $\ell \geq 1$, we have the bound
\[
   \Norm{u_t}_{\mc([0,\infty); L^\infty(\Omega_t))} \leq \Norm{v_t}_{\mc([0,\infty)\times\Omega_0)} \leq R_V
\]
derived in Lemma \ref{lem_spt}. 
This together with the assumption $|\psi'| \leq C|\psi|$ yields
\[
   \left| \partial_x \left( \frac{\int_{\R} \psi(x-y) u_{t-\tau}(y) \rho_{t-\tau}(y)\d y}{\int_{\R} \psi(x-y) \rho_{t-\tau}(y)\d y} \right) \right|
     \leq
     2CR_V = \overline C. 
\]
Consequently,
\( \label{ineq_d}
   |D_t w_t + w_t^2 + w_t | \leq \overline C,
\)
where the constant $\overline C>0$ is independent of time.
We can now distinguish the following two cases:
\begin{itemize}
\item \textbf{Subcritical case:}
Assume that $4\overline C \leq 1$. Then it follows from \eqref{ineq_d} that
\[
   D_t w_t \geq - (w_t^2 + w_t + \overline C) = - \left( w_t - w^{1,+})(w_t - w^{1,-} \right)
\]
with
\[
   w^{1,\pm} := \frac{-1 \pm \sqrt{1 - 4\overline C}}{2}.
\]
Consequently, if $w_0 \geq w^{1,-}$, then $w_t \geq w^{1,-}$ for all $t \geq 0$.
\vspace{2mm}

\item \textbf{Supercritical case:} Again, using \eqref{ineq_d}, we have
\[
   D_t w_t \leq -(w_t^2 + w_t - \overline C) = - (w_t - w^{2,+})(w_t - w^{2,-})
\]
with
\[
   w^{2,\pm} :=\frac{-1 \pm \sqrt{1 + 4\overline C}}{2}.
\]
Consequently, if $w_0 < w^{2,-}$, then $w_t \leq w^{2,-}$ for all $t \geq 0$.
Due to $w^{2,-} < w^{2,+}$, this further implies that
\[
   D_t w_t \leq -(w_t - w^{2,-})^2,
\]
and solving this differential inequality gives
\[
   w_t \leq \left( t + \frac{1}{w_0 - w^{2,-}} \right)^{-1} + w^{2,-}.
\]
Thus $w_t$ is only defined on a finite time interval
and diverges to $-\infty$ before $t = (w^{2,-} - w_0)^{-1}$.
\end{itemize}
Collecting the above observations completes the proof of Theorem \ref{thm_cri}.

%
%
\section*{Acknowledgments}
YPC was supported by National Research Foundation of Korea(NRF) grant funded by the Korea government(MSIP) (No. 2017R1C1B2012918) and the Alexander Humboldt Foundation through the Humboldt Research Fellowship for Postdoctoral Researchers.
JH was supported by KAUST baseline funds and KAUST grant no. 1000000193.

%
%

\end{document}